\documentclass[sn-mathphys,Numbered]{sn-jnl}

\usepackage{graphicx}%
\usepackage{multirow}%
\usepackage{amsmath,amssymb,amsfonts}%
\usepackage{amsthm}%
\usepackage{mathrsfs}%
\usepackage[title]{appendix}%
\usepackage{xcolor}%
\usepackage{textcomp}%
\usepackage{manyfoot}%
\usepackage{booktabs}%
\usepackage{algorithm}%
\usepackage{algorithmicx}%
\usepackage{algpseudocode}%
\usepackage{listings}%

\theoremstyle{thmstyleone}%

\newtheorem{theorem}{Theorem}[section]

\newtheorem{corollary}[theorem]{Corollary}

\theoremstyle{thmstyletwo}%
\newtheorem{example}[theorem]{Example}%

\theoremstyle{thmstylethree}%
\newtheorem{definition}[theorem]{Definition}%

\begin{document}

\title[On Statistical Convergence Of Order ${\alpha}$ In Partial Metric Spaces]{On Statistical Convergence Of Order ${\alpha}$ In Partial Metric Spaces}


\author*[1]{\fnm{Erdal} \sur{Bayram}}\email{ebayram@nku.edu.tr}

\author[2]{\fnm{\c{C}i\u{g}dem A.} \sur{Bekta\c{s}}\email{cbektas@firat.edu.tr}}
\equalcont{These authors contributed equally to this work.}

\author[3]{\fnm{Yavuz} \sur{Alt{\i}n}}\email{yaltin23@yahoo.com}
\equalcont{These authors contributed equally to this work.}

\affil*[1]{\orgdiv{Tekirda\u{g} Nam\i k Kemal University}, \orgname{Faculty of Science and Arts, Department of\ Mathematics,}, \postcode{59030}, \city{Tekirda\u{g}}, \country{Turkey}}

\affil[2]{\orgdiv{F\i rat University, Faculty of Science, Department of\ Mathematics, Elaz\i\u{g}, Turkey}}


\abstract{The present study introduces the notions of statistical
	convergence of order $\alpha$ and strong $p-$ Ces\`{a}ro summability of order 
	$\alpha$ in partial metric spaces. Also, we examine the inclusion relations
	between these concepts. In addition, we introduce the notion of $\lambda -$%
	statistical convergence of order $\alpha$ in partial metric spaces while
	providing relations linked to these sequence spaces.}




\maketitle

\section{Introduction}\label{sec1}

One of the two notions that constitute the concern of our study is the
partial metric space, which is a generalization of the usual metric spaces
as introduced by Matthews\cite{matthews1994}. The main difference between the partial
metric and the standard metric is that the distance of an arbitrary point
does not need to be equal to zero. The partial metric in which was first
used in computer science to be later applied to many other fields.

The second notion is statistical convergence, which is defined as a
generalization of sequential convergence. In the first edition of Zygmund's
monograph, published in Warsaw, the definition of statistical convergence
(as almost convergence) was given by Zygmund \cite{zygmund1979}. Steinhaus \cite{steinhaus1951} and
Fast \cite{fast1951} and later Schoenberg \cite{schoenberg1959} introduced the concept of
statistical convergence, independently.This subject has been studied by many
mathematicians, (for example, see Alt\i n et al. \cite{altin2015} , Bilalov and Nazarova \cite{bilalov2015}, Kayan et al.
\cite{kayan2018}, K\"{u}\c{c}\"{u}kaslan et al. \cite{kucukaslan2014}). The concept of statistical
convergence has been used in many areas of mathematics, such as Number
theory, Probabilistic normed spaces, Ergodic theory, Fourier analysis,
Measure theory, Trigonometric series, and others. Therefore, it is a very
active and intensively studied subject in many contexts.

Before proceeding with results, for the convenience of the reader, let us
recall the definitions and terminology that this work involves.

The definition of the partial metric space is as it follows:

Let $X$ be a non-empty set and $\rho:X$ $\times X\rightarrow\mathbb{R}$ be a
function such that for all $x,y,z\in X,$

$i)$ $0\leq\rho(x,x)\leq\rho(x,y),$

$ii)$ If $\rho(x,x)=\rho(x,y)=\rho(y,y)$ then $x=y,$

$iii)$ $\rho(x,y)=\rho(y,x),$

$iv)$ $\rho(x,z)\leq\rho(x,y)+$ $\rho(y,z)-\rho(y,y).$

Then $\rho$ is called a partial metric and the pair $(X,\rho)$ is called a
partial metric space (see Matthews \cite{matthews1994}).

We will now give the concepts of convergence and bounded in the partial
metric space.

Let $(X,\rho)$ be a partial metric space and $(x_{n})$ be a sequence in $X.$
Then

$i)$ $(x_{n})$ is bounded if there exists $a$ real number $M>0$ such that $%
\rho(x_{n},x_{m})\leq M$ for all $n,m\in\mathbb{N}$,

$ii)$ $\left( x_{n}\right) $ is called convergent to $x$ in $(X,\rho),$
written as $\lim\limits_{n\rightarrow\infty}x_{n}=x,$ if 
\begin{equation*}
	\lim\limits_{n\rightarrow\infty}\rho(x_{n},x)=\lim\limits_{n\rightarrow%
		\infty }\rho(x_{n},x_{n})=\rho(x,x),\text{see Nuray \cite{Nuray2022}.}
\end{equation*}

The statistical convergence depends on density of subsets of $%
\mathbb{N}
$. The natural density of $K\subset 
\mathbb{N}
$ which is the main tool for this convergence is defined by 
\begin{equation*}
	\delta \left( K\right) =\lim_{n\rightarrow \infty }\frac{1}{n}\left\vert
	\left\{ k\leq n:k\in K\right\} \right\vert ,
\end{equation*}%
where $\left\vert \left\{ k\leq n:k\in K\right\} \right\vert $ denotes the
number of elements of $K$ in which does not exceed $n$ (see Fridy \cite{Fridy1985}).
Clearly, any finite subset of $%
\mathbb{N}
$ has zero natural density and $\delta \left( K^{c}\right) =1-\delta \left(
K\right) $. For a detailed description of the density of subsets of $\mathbb{%
	N}$, reference can be made to Niven and Zuckerman \cite{Niven1966}.

A sequence $\left( x_{k}\right) $ of complex numbers is said to be
statistically convergent to a number $L$ if for every positive number $%
\varepsilon,$ $\delta\left( \left\{ k\in\mathbb{N}\mathbf{:}\text{ }%
\left\vert x_{k}-L\right\vert \geq\varepsilon\right\} \right) $ has natural
density zero. The number $L$ is called statistical limit of $\left(
x_{k}\right) $ and is written as $S-\lim x_{k}=L$ or $x_{k}\rightarrow
L\left( S\right) $ and the set of all statistically convergent sequences is denoted by $%
S.$

Leindler \cite{leindler1965} introduced $(V,\lambda )$-summability by the help of sequence 
$\lambda =\left( \lambda _{n}\right) $ as in the following: Let $\lambda
=\left( \lambda _{n}\right) $ be a non-decreasing sequence of positive
numbers tending to $\infty $ with $\lambda _{n+1}\leq \lambda _{n}+$ $1$, $%
\lambda _{1}=1$. The generalized de la Vallee-Poussin mean is defined by 
\begin{equation*}
	t_{n}\left( x\right) =\frac{1}{\lambda _{n}}\sum\limits_{k\in I_{n}}x_{k}
\end{equation*}%
where $I_{n}=[n-\lambda _{n}+1,n]$. Accordingly, a sequence $\left(
x_{k}\right) $ of numbers is said to be $(V,\lambda )$-summable to a number $%
L$ if $t_{n}\left( x\right) \rightarrow L$ as $n\rightarrow \infty $. 
\begin{equation*}
	\lbrack C,1]=\left\{ \left( x_{k}\right) :\exists L\in 
	\mathbb{R}
	,\lim_{n\rightarrow \infty }\frac{1}{n}\sum\limits_{k=1}^{n}\left\vert
	x_{k}-L\right\vert =0\right\}
\end{equation*}%
and%
\begin{equation*}
	\lbrack V,\lambda ]=\left\{ \left( x_{k}\right) :\exists L\in 
	\mathbb{R}
	,\lim_{n\rightarrow \infty }\frac{1}{\lambda _{n}}\sum\limits_{k\in
		I_{n}}\left\vert x_{k}-L\right\vert =0\right\}
\end{equation*}%
denote the sets of sequences $\left( x_{k}\right) $ which are strongly Ces%
\'{a}ro summable and strongly $(V,\lambda )$-summable to $L$. It is noted
that for $\lambda _{n}=n$, $\left( V,\lambda \right) $-summability reduces
to $\left( C,1\right) $-summability.

Mursaleen \cite{mursaleen2000} introduced the $\lambda $-density of $K\subset 
\mathbb{N}
$ as defined by 
\begin{equation*}
	\delta _{\lambda }\left( K\right) =\lim_{n\rightarrow \infty }\frac{1}{%
		\lambda _{n}}\left\vert \left\{ n-\lambda _{n}+1\leq k\leq n:k\in K\right\}
	\right\vert
\end{equation*}%
and $\lambda $-statistical convergence as it follows:

A sequence $\left( x_{k}\right) $ of numbers is said to be $\lambda $%
-statistically convergent to a number $L$ provided that for every $%
\varepsilon>0$,%
\begin{equation*}
	\lim_{n\rightarrow\infty}\frac{1}{\lambda_{n}}\left\vert \left\{
	n-\lambda_{n}+1\leq k\leq n:\left\vert x_{k}-L\right\vert \geq\varepsilon
	\right\} \right\vert =0\text{.}
\end{equation*}
In this case, the number $L$ is called $\lambda$-statistical limit of the
sequence $\left( x_{k}\right) $.

The statistical convergence with degree $0<\beta <1$ was introduced by
Gadjiev and Orhan \cite{gadjiev2002}. The statistical convergence of order $\alpha $ and
strong $p$-Ces\`{a}ro summability of order $\alpha $ were later studied by 
\c{C}olak \cite{colak2010}. Also \c{C}olak and Bekta\c{s} \cite{colakbektas2011} introduced $\lambda $%
-statistical convergence of order $\alpha $, as in the following:

Let the sequence $\lambda=\left( \lambda_{n}\right) $ of real numbers be
defined as above and $0<\alpha\leq1$. The sequence $x=\left( x_{k}\right)
\in w$ is said to be $\lambda$-statistically convergent of order $\alpha$ if
there is a complex number $L$ such that

\begin{equation*}
	\lim_{n\rightarrow\infty}\frac{1}{\lambda_{n}^{\alpha}}\left\vert \left\{
	k\in I_{n}:\left\vert x_{k}-L\right\vert \geq\varepsilon\right\} \right\vert
	=0\text{,}
\end{equation*}
where $I_{n}=[n-\lambda_{n}+1,$ $n]$ and $\lambda_{n}^{\alpha}$ is the
coordinates of $\alpha$ th power of the sequence $\lambda$, that is, $%
\lambda^{\alpha}=\left( \lambda_{n}^{\alpha}\right) =\left( \lambda
_{1}^{\alpha},\lambda_{2}^{\alpha},...,\lambda_{n}^{\alpha},...\right) $.

Some new sequence spaces for $\lambda$ and $\mu$ different sequences of $%
\Lambda$ class were later defined by \c{C}olak \cite{colak2011} and some inclusion
theorems were examined.

In this study, we denote the class of all decreasing sequence of positive
real numbers tending to $\infty,$ such that $\lambda_{n+1}\leq\lambda_{n}+1$%
, $\lambda_{1}=1$ by $\Lambda$. Also, unless stated otherwise, by "for all $%
n\in%
\mathbb{N}
_{n_{0}}$" we mean "for all $n\in%
\mathbb{N}
$ except finite numbers of positive integers", where $%
\mathbb{N}
_{n_{0}}=\left\{ n_{0},n_{0}+1,n_{0}+2,...\right\} $ for some $n_{0}\in%
\mathbb{N}
=\{1,2,3,...\}$.

The concept of statistical convergence in partial metric spaces was given by
Nuray \cite{Nuray2022} as it follows:

Let $\left( x_{k}\right) $ be a sequence in partial metric space $(X,\rho)$.
The sequence $\left( x_{k}\right) $ is said to be $\rho-$ statistically
convergent to $x$ if there exists a point $L\in X$ such that

\begin{equation*}
	\underset{n\rightarrow\infty}{\lim}\frac{1}{n}\left\vert \left\{ k\leq
	n:\left\vert \rho\left( x_{k},x\right) -\rho\left( x,x\right) \right\vert
	\geq\varepsilon\right\} \right\vert =0
\end{equation*}
for every $\varepsilon>0$. In addition, Nuray \cite{Nuray2022} also examined the
relationship between the concept of statistical convergence in partial
metric spaces and strong Ces\`{a}ro summability.

\section{Statistical Convergence of Order $\alpha$ in Partial Metric Spaces}\label{sec2}

In this section, we introduce the notions of statistical convergence of
order $\alpha$ and strong $p-$ Ces\`{a}ro summability of order $\alpha$ in
partial metric spaces. Also, some relations between statistical
convergence of order $\alpha$ and strongly $p-$ Ces\`{a}ro summable sequences of order 
$\alpha$ are given.

\begin{definition} 
For given\textbf{\ }a real $\alpha\in\left( 0,1%
\right] ,$ the sequence $\left( x_{k}\right) $ in the partial metric space $%
(X,\rho)$ is said to be $\rho-$ statistically convergent of order $\alpha$,
if there exists a point $x\in X$ such that

\begin{equation*}
	\underset{n\rightarrow\infty}{\lim}\frac{1}{n^{\alpha}}\left\vert \left\{
	k\leq n:\left\vert \rho\left( x_{k},x\right) -\rho\left( x,x\right)
	\right\vert \geq\varepsilon\right\} \right\vert =0
\end{equation*}
for every $\varepsilon>0$. In this case, it is stated that $\left(
x_{k}\right) $ is $\rho-$ statistically convergent of order $\alpha$ to $x$
which is denoted by $S_{\rho}^{\alpha}-\lim x_{k}=x$ or $x_{k}\rightarrow
x\left( S_{\rho}^{\alpha}\left( X\right) \right) $.
\end{definition}
Throughout this paper, $S_{\rho}^{\alpha}(X)$ will denote the class of
sequences in partial metric space $(X,\rho)$ which are $\rho-$ statistically
convergent of order $\alpha$.

For the sake of simplicity, it will be considered that the sequence $\left(
x_{k}\right) $ and the element $x$ we use in the proofs are chosen from the
partial metric space $(X,\rho)$, although we do not emphasize it every time.

\begin{theorem}
For some reals\textbf{\ }$\alpha$ and $\beta$ such that $0<\alpha<\beta\leq1$, the inclusion $S_{\rho}^{\alpha}(X)\subseteq S_{\rho}^{\beta}(X)$ holds.
\end{theorem}

\begin{proof}
	Suppose that $0<\alpha<\beta\leq1$. Then, the inequality%
	\begin{equation*}
		0\leq\underset{n\rightarrow\infty}{\lim}\frac{1}{n^{\beta}}\left\vert
		\left\{ k\leq n:\left\vert \rho\left( x_{k},x\right) -\rho\left( x,x\right)
		\right\vert \geq\varepsilon\right\} \right\vert \leq\underset{n\rightarrow
			\infty}{\lim}\frac{1}{n^{\alpha}}\left\vert \left\{ k\leq n:\left\vert
		\rho\left( x_{k},x\right) -\rho\left( x,x\right) \right\vert
		\geq\varepsilon\right\} \right\vert
	\end{equation*}
	is provided for $\left( x_{k}\right) \subset X$, $x\in X$ and for every $%
	\varepsilon>0$ and this clearly gives desired the inclusion $%
	S_{\rho}^{\alpha}(X)\subseteq S_{\rho}^{\beta}(X).$
\end{proof}
That the inclusion may be strict can be seen by the following example.
\begin{example}
Let us consider the partial metric of real numbers
defined $\rho:%
\mathbb{R}
\times%
\mathbb{R}
\rightarrow%
\mathbb{R}
,\rho\left( x,y\right) =\max\left\{ x,y\right\} $, and the sequence $\left(
x_{k}\right) \subset%
\mathbb{R}
$ such that 
\begin{equation*}
	x_{k}=\left\{ 
	\begin{array}{c}
		1\text{, \ \ \ \ }k\text{ is a square} \\ 
		0\text{, \ \ \ \ \ \ \ \ \ otherwise}%
	\end{array}
	\right. \text{.}
\end{equation*}
Clearly, for $\varepsilon>0,$ $\left\vert \left\{ k\leq n:\left\vert
\rho\left( x_{k},0\right) -\rho\left( 0,0\right) \right\vert
\geq\varepsilon\right\} \right\vert \leq\sqrt{n}$ holds. This means that $%
\left( x_{k}\right) \in$ $S_{\rho}^{\beta}(\lambda,X)$ for $\frac{1}{2}%
<\beta\leq1$ but $\left( x_{k}\right) \notin$ $S_{\rho}^{\alpha}(\lambda,X)$
for $0<\alpha\leq\frac{1}{2}$, that is $S_{\rho}^{\alpha }(X)\subset
S_{\rho}^{\beta}(X)$.
\end{example}
Taking $\beta=1$, we get the following result from the last inequality above.

\begin{corollary}
For any $0<\alpha\leq1$, if a sequence is $\rho-$
statistically convergent of order $\alpha$ to $x,$ then it is $\rho -$%
statistically convergent to $x$, in other words, $S_{\rho}^{\alpha
}(X)\subseteq S_{\rho}(X)$.
\end{corollary}

From Theorem 2.2 we have the following results in which the proofs are easy.

\begin{corollary}
For $\alpha,\beta\in\left( 0,1\right] $, the
following statements hold:

$i)$ $S_{\rho}^{\alpha}(X)=S_{\rho}^{\beta}(X)$ $\Longleftrightarrow
\alpha=\beta,$

$ii)$ $S_{\rho}^{\alpha}(X)=S_{\rho}(X)$ $\Longleftrightarrow\alpha=1.$
\end{corollary}

\begin{definition}
Let $q$ be a positive real number and $\alpha
\in\left( 0,1\right] $. Then, in a partial metric space $(X,\rho)$, we say
that the sequence $\left( x_{k}\right) $ is strongly $q-$ Ces\`{a}ro
summable of order $\alpha$ to $x\in X$ if 
\begin{equation*}
	\underset{n\rightarrow\infty}{\lim}\frac{1}{n^{\alpha}}\sum \limits 
	_{\substack{ k=1}}^{n}\left\vert \rho\left( x_{k},x\right) -\rho\left(
	x,x\right) \right\vert ^{q}=0.
\end{equation*}
In this case, we write $\left[ C,q\right] -\lim x_{k}=x.$
\end{definition}

In the following theorem, we examine the relationship between statistical
convergence and Ces\`{a}ro convergence in partial metric spaces.

\begin{theorem}
Let $\alpha$ and $\theta$ be fixed real numbers such
that $0<\alpha<\theta\leq1$ and $0<q<\infty.$ In the partial metric space $%
(X,\rho)$, if a sequence is strongly $p-$ Cesaro summable of order $\alpha$
to $x\in X,$ then it is statistically convergent of order $\theta$ to $x.$
\end{theorem}

\begin{proof}
	For any $\varepsilon>0,$ we have 
	\begin{align*}
		& \underset{n\rightarrow\infty}{\lim}\frac{1}{n^{\alpha}}\sum
		\limits_{k=1}^{n}\left\vert \rho\left( x_{k},x\right) -\rho\left( x,x\right)
		\right\vert ^{q} \\
		& =\underset{n\rightarrow\infty}{\lim}\frac{1}{n^{\alpha}}\left( \sum
		\limits_{\substack{ _{\substack{ k=1}}  \\ \left\vert \rho\left(
				x_{k},x\right) -\rho\left( x,x\right) \right\vert \geq\varepsilon}}%
		^{n}\left\vert \rho\left( x_{k},x\right) -\rho\left( x,x\right) \right\vert
		^{q}+\sum \limits_{\substack{ _{\substack{ k=1}}  \\ \left\vert \rho\left(
				x_{k},x\right) -\rho\left( x,x\right) \right\vert <\varepsilon}}%
		^{n}\left\vert \rho\left( x_{k},x\right) -\rho\left( x,x\right) \right\vert
		^{q}\right) \\
		& \geq\underset{n\rightarrow\infty}{\lim}\frac{1}{n^{\alpha}}\sum \limits 
		_{\substack{ _{\substack{ k=1}}  \\ \left\vert \rho\left( x_{k},x\right)
				-\rho\left( x,x\right) \right\vert \geq\varepsilon}}^{n}\left\vert
		\rho\left( x_{k},x\right) -\rho\left( x,x\right) \right\vert ^{q} \\
		& \geq\underset{n\rightarrow\infty}{\lim}\frac{1}{n^{\alpha}}\left\vert
		\left\{ k\leq n:\left\vert \rho\left( x_{k},x\right) -\rho\left( x,x\right)
		\right\vert \geq\varepsilon\right\} \right\vert \varepsilon^{q} \\
		& \geq\left( \underset{n\rightarrow\infty}{\lim}\frac{1}{n^{\theta}}%
		\left\vert \left\{ k\leq n:\left\vert \rho\left( x_{k},x\right) -\rho\left(
		x,x\right) \right\vert \geq\varepsilon\right\} \right\vert \right)
		\varepsilon^{q}\geq0.
	\end{align*}
	This reveals that if the sequence $\left( x_{k}\right) $ is strongly $p-$
	Cesaro summable of order $\alpha$ to $x\in X,$ then it is statistically
	convergent of order $\theta$ to $x.$
\end{proof}
If we take $\theta=\alpha$ in Theorem 2.7, we obtain the following result:

\begin{corollary}
\emph{COROLLARY 2.8 }Let $\alpha$ be a fixed real number such that $%
0<\alpha\leq1$ and $0<q<\infty$. If a sequence in the partial metric space $%
(X,\rho)$ is strongly $p-$ Cesaro summable of order $\alpha$ to $x\in X,$
then it is statistically convergent of order $\alpha$ to $x.$
\end{corollary}

Hence, from Corollary 2.8, we have the necessary part of Theorem 4.4 in\
Nuray \cite{Nuray2022} in case $\alpha =1$ as if a sequence is strongly $p-$ Cesaro
summable to $x\in X,$ then it is statistically convergent to $x\in X$.

\section{$\lambda -$ Statistical Convergence of Order $\alpha $ in Partial Metric Spaces}\label{sec3}

In this section, we introduce the notion of $\lambda-$statistical
convergence of order $\alpha$ and $[V,\lambda]-$summability of order $\alpha$
in partial metric spaces. In this setting, some inclusion results related
these concepts are also included in this section.

\begin{definition}
For given $\lambda=\left( \lambda_{n}\right)
\in\Lambda$ and $\alpha\in\left( 0,1\right] $, the sequence $\left(
x_{k}\right) $ in the partial metric space $(X,\rho)$ is said to be $%
\lambda\rho-$ statistically convergent of order $\alpha$, if there exists a
point $x\in X$ such that

\begin{equation*}
	\underset{n\rightarrow\infty}{\lim}\frac{1}{\lambda_{n}^{\alpha}}\left\vert
	\left\{ k\in I_{n}:\left\vert \rho\left( x_{k},x\right) -\rho\left(
	x,x\right) \right\vert \geq\varepsilon\right\} \right\vert =0
\end{equation*}
holds for every $\varepsilon>0$. In this case, we write $S_{\rho}^{\alpha
}(\lambda)-\lim x_{k}=x$ or $x_{k}\rightarrow x\left( S_{\rho}^{\alpha
}(\lambda)\right) $ and $S_{\rho}^{\alpha}(\lambda,X)$ will denote the class
of all $\lambda\rho-$ statistically convergent sequences which are $%
\lambda\rho-$ statistically convergent of order $\alpha$ in the partial
metric space $(X,\rho).$
\end{definition}

\begin{theorem}
Let $0<\alpha\leq\beta\leq1.$ Then $S_{\rho}^{\alpha
}(\lambda,X)\subset S_{\rho}^{\beta}(\lambda,X)$ for some $\alpha$ and $%
\beta $ such that $\alpha<\beta.$
\end{theorem}

\begin{proof}
	If $0<\alpha\leq\beta\leq1,$ then 
	\begin{equation*}
		0\leq\underset{n\rightarrow\infty}{\lim}\frac{1}{\lambda_{n}^{\beta}}%
		\left\vert \left\{ k\in I_{n}:\left\vert \rho\left( x_{k},x\right)
		-\rho\left( x,x\right) \right\vert \geq\varepsilon\right\} \right\vert \leq%
		\text{ }\underset{n\rightarrow\infty}{\lim}\frac{1}{\lambda_{n}^{\alpha}}%
		\left\vert \left\{ k\in I_{n}:\left\vert \rho\left( x_{k},x\right)
		-\rho\left( x,x\right) \right\vert \geq\varepsilon\right\} \right\vert
	\end{equation*}

	for every $\varepsilon>0$ and this gives $S_{\rho}^{\alpha}(\lambda,X)%
	\subset S_{\rho}^{\beta}(\lambda,X).$
\end{proof}

The following example states that the inclusion in the previous theorem may
be strict.

\begin{example}
Let us consider the natural partial metric of real
numbers, $\rho:%
\mathbb{R}
\times%
\mathbb{R}
\rightarrow%
\mathbb{R}
,\rho\left( x,y\right) =-\min\left\{ x,y\right\} $, and the sequence $%
x=\left( x_{k}\right) \subset%
\mathbb{R}
$ such that 
\begin{equation*}
	x_{k}=\left\{ 
	\begin{array}{c}
		k\text{, }n-\sqrt{\lambda_{n}}+1\leq k\leq n \\ 
		0\text{, \ \ \ \ \ \ \ \ \ \ \ \ \ \ \ \ \ otherwise}%
	\end{array}
	\right. \text{.}
\end{equation*}
From \c{C}olak \cite{colak2010}, it is easily seen that $x\in$ $S_{\rho}^{\beta}(%
\lambda,X)$ for $\frac{1}{2}<\beta\leq1$ but $x\notin$ $S_{\rho}^{\alpha
}(\lambda,X)$ for $0<\alpha\leq\frac{1}{2}$.
\end{example}

\begin{corollary}
\ \medskip

$i)$ $S_{\rho}^{\alpha}(\lambda,X)=S_{\rho}^{\beta}(\lambda,X)$ $%
\Longleftrightarrow\alpha=\beta,$

$ii)$ $S_{\rho}^{\alpha}(\lambda,X)=S_{\rho}(\lambda,X)$ $%
\Longleftrightarrow \alpha=1$
\end{corollary}

For $\alpha\in\left( 0,1\right] $ the inclusion $S_{\rho}^{\alpha}(%
\lambda,X)\subseteq S_{\rho}^{\alpha}(X)$ clearly hold. The following
theorem gives a case where the reverse of this inclusion is also holds.

\begin{theorem}
For $\alpha\in\left( 0,1\right] $ the inclusion$S_{\rho
}^{\alpha}(X)\subseteq S_{\rho}^{\alpha}(\lambda,X)$\textbf{\ }holds\textbf{%
	\ }if 
\begin{equation}
	\lim_{n\rightarrow\infty}\inf\frac{\lambda_{n}^{\alpha}}{n^{\alpha}}>0. 
	\tag{3.1}
\end{equation}
\medskip
\end{theorem}

\begin{proof}
	For a given $\varepsilon>0,$ we obtain that 
	\begin{equation*}
		\left\vert \left\{ k\leq n:\left\vert \rho\left( x_{k},x\right) -\rho\left(
		x,x\right) \right\vert \geq\varepsilon\right\} \right\vert \geq\left\vert
		\left\{ k\in I_{n}:\left\vert \rho\left( x_{k},x\right) -\rho\left(
		x,x\right) \right\vert \geq\varepsilon\right\} \right\vert .
	\end{equation*}
	Therefore, 
	\begin{align*}
		\underset{n\rightarrow\infty}{\lim}\frac{1}{n^{\alpha}}\left\vert \left\{
		k\leq n:\left\vert \rho\left( x_{k},x\right) -\rho\left( x,x\right)
		\right\vert \geq\varepsilon\right\} \right\vert & \geq\text{ }\underset{%
			n\rightarrow\infty}{\lim}\frac{1}{n^{\alpha}}\left\vert \left\{ k\in
		I_{n}:\left\vert \rho\left( x_{k},x\right) -\rho\left( x,x\right)
		\right\vert \geq\varepsilon\right\} \right\vert \\
		& =\frac{\lambda_{n}^{\alpha}}{n^{\alpha}}.\frac{1}{\lambda_{n}^{\alpha}}%
		\left\vert \left\{ k\in I_{n}:\left\vert \rho\left( x_{k},x\right)
		-\rho\left( x,x\right) \right\vert \geq\varepsilon\right\} \right\vert \geq0
	\end{align*}
	Taking limit as $n\rightarrow\infty$ and using (3.1), we get $%
	x_{k}\rightarrow x\left( S_{\rho}^{\alpha}(X)\right) $ implies $%
	x_{k}\rightarrow x\left( S_{\rho}^{\alpha}(\lambda,X)\right) .$
\end{proof}

\begin{theorem}
Let $\lambda=\left( \lambda_{n}\right) $ and $%
\mu=\left( \mu_{n}\right) $ belong to $\Lambda$ such that $%
\lambda_{n}\leq\mu_{n}$ for all $n\in%
\mathbb{N}
_{n_{0}}$ and let $\alpha$ and $\beta$ be such that $0<\alpha\leq\beta\leq1$%
. Then, the following statements hold:

$\left( i\right) $ If 
\begin{equation}
	\lim_{n\rightarrow\infty}\inf\frac{\lambda_{n}^{\alpha}}{\mu_{n}^{\beta}}>0 
	\tag{3.2}  \label{inf1}
\end{equation}

then $S_{\rho}^{\beta}(\mu,X)\subseteq S_{\rho}^{\alpha}(\lambda,X)$

$\left( ii\right) $ If 
\begin{equation}
	\lim_{n\rightarrow\infty}\frac{\lambda_{n}^{\alpha}}{\mu_{n}^{\beta}}=1\text{
		and }\lim_{n\rightarrow\infty}\frac{\mu_{n}}{\mu_{n}^{\beta}}=1  \tag{3.3}
	\label{inf2}
\end{equation}

then $S_{\rho}^{\alpha}(\lambda,X)=$ $S_{\rho}^{\beta}(\mu,X)$.
\end{theorem}

\begin{proof}
	$\left( i\right) $ Supposed that $\lambda_{n}\leq\mu_{n}$
	for all $n\in%
	\mathbb{N}
	_{n_{0}}$ and (\ref{inf1}) is satisfied. Since $I_{n}\subset J_{n}$, for
	given $\varepsilon>0,$ we have 
	\begin{equation*}
		\left\{ k\in J_{n}:\left\vert \rho\left( x_{k},x\right) -\rho\left(
		x,x\right) \right\vert \geq\varepsilon\right\} \supset\left\{ k\in
		I_{n}:\left\vert \rho\left( x_{k},x\right) -\rho\left( x,x\right)
		\right\vert \geq\varepsilon\right\}
	\end{equation*}
	
	and so 
	\begin{equation*}
		\frac{1}{\mu_{n}^{\beta}}\left\vert \left\{ k\in J_{n}:\left\vert \rho\left(
		x_{k},x\right) -\rho\left( x,x\right) \right\vert \geq\varepsilon\right\}
		\right\vert \geq\frac{\lambda_{n}^{\alpha}}{\mu_{n}^{\beta}}.\frac{1}{%
			\lambda_{n}^{\alpha}}\left\vert \left\{ k\in I_{n}:\left\vert \rho\left(
		x_{k},x\right) -\rho\left( x,x\right) \right\vert \geq\varepsilon\right\}
		\right\vert
	\end{equation*}
	
	for all $n\in%
	\mathbb{N}
	_{n_{0}}$, where $J_{n}=[n-\mu_{n}+1,n]$. By (\ref{inf1}) as $n\rightarrow
	\infty$\ we get $S_{\rho}^{\beta}(\mu,X)\subseteq
	S_{\rho}^{\alpha}(\lambda,X)$.
	
	$\left( ii\right) $ Supposed that $\left( x_{k}\right) \in S_{\rho
	}^{\alpha}(\lambda,X)$ and (\ref{inf2}) hold. Since $I_{n}\subset J_{n}$,
	for $\varepsilon>0$ we have, for all $n\in%
	\mathbb{N}
	_{n_{0}},$ 
	\begin{align*}
		0 & \leq\frac{1}{\mu_{n}^{\beta}}\left\vert \left\{ k\in J_{n}:\left\vert
		\rho\left( x_{k},x\right) -\rho\left( x,x\right) \right\vert
		\geq\varepsilon\right\} \right\vert \\
		& =\frac{1}{\mu_{n}^{\beta}}\left\vert \left\{ n-\mu_{n}+1\leq
		k<n-\lambda_{n}+1:\left\vert \rho\left( x_{k},x\right) -\rho\left(
		x,x\right) \right\vert \geq\varepsilon\right\} \right\vert \\
		& \text{ \ \ \ \ }+\frac{1}{\mu_{n}^{\beta}}\left\vert \left\{ k\in
		I_{n}:\left\vert \rho\left( x_{k},x\right) -\rho\left( x,x\right)
		\right\vert \geq\varepsilon\right\} \right\vert \\
		& \leq\left( \frac{\mu_{n}-\lambda_{n}}{\mu_{n}^{\beta}}\right) +\frac {1}{%
			\mu_{n}^{\beta}}\left\vert \left\{ k\in I_{n}:\left\vert \rho\left(
		x_{k},x\right) -\rho\left( x,x\right) \right\vert \geq\varepsilon\right\}
		\right\vert \\
		& \leq\left( \frac{\mu_{n}-\lambda_{n}^{\alpha}}{\mu_{n}^{\beta}}\right) +%
		\frac{1}{\mu_{n}^{\beta}}\left\vert \left\{ k\in I_{n}:\left\vert \rho\left(
		x_{k},x\right) -\rho\left( x,x\right) \right\vert \geq\varepsilon\right\}
		\right\vert \\
		& \leq\left( \frac{\mu_{n}}{\mu_{n}^{\beta}}-\frac{\lambda_{n}^{\alpha}}{%
			\mu_{n}^{\beta}}\right) +\frac{1}{\lambda_{n}^{\alpha}}\left\vert \left\{
		k\in I_{n}:\left\vert \rho\left( x_{k},x\right) -\rho\left( x,x\right)
		\right\vert \geq\varepsilon\right\} \right\vert
	\end{align*}
	Thus, since $\underset{n}{\lim}\frac{\mu_{n}}{\mu_{n}^{\beta}}=1$ and $%
	\underset{n}{\lim}\frac{\lambda_{n}^{\alpha}}{\mu_{n}^{\beta}}=1$ by (\ref%
	{inf2}), first part of the sum in the right hand side of last inequality
	tend to $0$ as $n\rightarrow\infty$, because of that $\frac{\mu_{n}}{\mu
		_{n}^{\alpha}}-\frac{\lambda_{n}^{\alpha}}{\mu_{n}^{\beta}}\geq0$ for all $%
	n\in%
	\mathbb{N}
	_{n_{0}}$. This means that $S_{\rho}^{\alpha}(\lambda,X)\subset$ $S_{\rho
	}^{\beta}(\mu,X)$. On the other hand, since (\ref{inf2}) implies (\ref{inf1}%
	) we have $S_{\rho}^{\alpha}(\lambda,X)=$ $S_{\rho}^{\beta}(\mu,X)$.
\end{proof}

By choosing the value of $\beta$, as a result of Theorem 3.6, we can give
the following two results.

\begin{corollary}
Let $\lambda=\left( \lambda_{n}\right) $ and $%
\mu=\left( \mu_{n}\right) $ belong to $\Lambda$ such that $%
\lambda_{n}\leq\mu_{n}$ for all $n\in%
\mathbb{N}
_{n_{0}}$ and (\ref{inf1}) holds. Then the following statements hold:

$\left( i\right) $ $S_{\rho}^{\alpha}(\mu,X)\subseteq S_{\rho}^{\alpha
}(\lambda,X)$ holds for each $\alpha\in(0,1]$.

$\left( ii\right) $ $S_{\rho}(\mu,X)\subseteq S_{\rho}^{\alpha}(\lambda,X)$
holds for each $\alpha\in(0,1]$.
\end{corollary}

\begin{corollary}
Let $\lambda=\left( \lambda_{n}\right) $ and $%
\mu=\left( \mu_{n}\right) $ belong to $\Lambda$ such that $%
\lambda_{n}\leq\mu_{n}$ for all $n\in%
\mathbb{N}
_{n_{0}}$ and (\ref{inf2}) holds. Then the following statements hold:

$\left( i\right) $ $S_{\rho}^{\alpha}(\lambda,X)\subseteq S_{\rho}^{\alpha
}(\mu,X)$ for each $\alpha\in(0,1]$,

$\left( ii\right) $ $S_{\rho}^{\alpha}(\lambda,X)\subseteq S_{\rho}(\mu,X)$
for each $\alpha\in(0,1]$.
\end{corollary}

Now, we introduce $[V_{\rho}^{\alpha},\lambda]-$summability of order $\alpha$
for the partial metric spaces.

\begin{definition}
For any $\alpha\in\left( 0,1\right] $, the sequence $%
\left( x_{k}\right) $ in the partial metric space $(X,\rho)$ is called
strongly $[V_{\rho}^{\alpha},\lambda]-$summable of order $\alpha$ to $x\in X$
if%
\begin{equation*}
	\lim_{n\rightarrow\infty}\frac{1}{\lambda_{n}^{\alpha}}\sum \limits_{k\in
		I_{n}}\left\vert \rho\left( x_{k},x\right) -\rho\left( x,x\right)
	\right\vert =0\text{.}
\end{equation*}
holds. This is indicated by $[V_{\rho}^{\alpha},\lambda]-\lim x_{k}=x$ and
the set of all sequences with $[V_{\rho}^{\alpha},\lambda]-$summable of
order $\alpha$ is denoted by $[V_{\rho}^{\alpha},\lambda]$.
\end{definition}

\begin{theorem}
For given $\lambda=\left( \lambda_{n}\right) $, $%
\mu=\left( \mu_{n}\right) \in\Lambda,$ assume that $\lambda_{n}\leq\mu_{n}$
for all $n\in%
\mathbb{N}
_{n_{0}}$ and $0<\alpha\leq\beta\leq1$ holds. Then, the following statements
hold:

$\left( i\right) $ If (\ref{inf1}) holds, then $[V_{\rho}^{\beta},\mu]%
\subseteq\lbrack V_{\rho}^{\alpha},\lambda]$.

$\left( ii\right) $ If (\ref{inf2}) holds, then $[V_{\rho}^{\beta},\mu]=[V_{%
	\rho}^{\alpha},\lambda]$.
\end{theorem}

\begin{proof}
	$\left( i\right) $ Suppose that $\lambda_{n}\leq\mu_{n}$ for
	all $n\in%
	\mathbb{N}
	_{n_{0}}$. Clearly, $I_{n}\subset J_{n}$ holds so that we may write, for all 
	$n\in%
	\mathbb{N}
	_{n_{0}}$, 
	\begin{equation*}
		\frac{1}{\mu_{n}^{\beta}}\sum \limits_{k\in J_{n}}\left\vert \rho\left(
		x_{k},x\right) -\rho\left( x,x\right) \right\vert \geq\frac{1}{%
			\mu_{n}^{\beta}}\sum \limits_{k\in I_{n}}\left\vert \rho\left(
		x_{k},x\right) -\rho\left( x,x\right) \right\vert \geq0
	\end{equation*}
	This gives 
	\begin{equation*}
		\frac{1}{\mu_{n}^{\beta}}\sum \limits_{k\in J_{n}}\left\vert \rho\left(
		x_{k},x\right) -\rho\left( x,x\right) \right\vert \geq\frac{%
			\lambda_{n}^{\alpha}}{\mu_{n}^{\beta}}\frac{1}{\lambda_{n}^{\alpha}}\sum
		\limits_{k\in I_{n}}\left\vert \rho\left( x_{k},x\right) -\rho\left(
		x,x\right) \right\vert \geq0\text{.}
	\end{equation*}
	\nolinebreak
	
	Hence, as $n\rightarrow\infty$ by (\ref{inf1}) we have $[V_{\rho}^{\beta},%
	\mu]\subset$ $[V_{\rho}^{\alpha},\lambda]$.
	
	$\left( ii\right) $ Suppose that $\left( x_{k}\right) \in\lbrack V_{\rho
	}^{\alpha},\lambda]$ and (\ref{inf2}) holds. Since $\left( x_{k}\right) $ is
	bounded, there exists some $M>0$ such that $\left\vert \rho\left(
	x_{k},x\right) -\rho\left( x,x\right) \right\vert \leq M$ for all $k\in%
	\mathbb{N}
	$. Now, since $\lambda_{n}\leq\mu_{n}$, so is $\frac{1}{\mu_{n}^{\beta}}\leq$
	$\frac{1}{\lambda_{n}^{\alpha}}$, and $I_{n}\subset J_{n}$ for each $n\in%
	\mathbb{N}
	_{n_{0}}$, we have, for every $n\in%
	\mathbb{N}
	_{n_{0}},$ 
	\begin{align*}
		0 & \leq\frac{1}{\mu_{n}^{\beta}}\sum \limits_{k\in J_{n}}\left\vert
		\rho\left( x_{k},x\right) -\rho\left( x,x\right) \right\vert =\frac{1}{%
			\mu_{n}^{\beta}}\sum \limits_{k\in J_{n}-I_{n}}\left\vert \rho\left(
		x_{k},x\right) -\rho\left( x,x\right) \right\vert +\frac{1}{\mu_{n}^{\beta}}%
		\sum \limits_{k\in I_{n}}\left\vert \rho\left( x_{k},x\right) -\rho\left(
		x,x\right) \right\vert \\
		& \leq\left( \frac{\mu_{n}-\lambda_{n}}{\mu_{n}^{\beta}}\right) M+\frac {1}{%
			\mu_{n}^{\beta}}\sum \limits_{k\in I_{n}}\left\vert \rho\left(
		x_{k},x\right) -\rho\left( x,x\right) \right\vert \\
		& \leq\left( \frac{\mu_{n}-\lambda_{n}^{\alpha}}{\mu_{n}^{\beta}}\right) M+%
		\frac{1}{\mu_{n}^{\beta}}\sum \limits_{k\in I_{n}}\left\vert \rho\left(
		x_{k},x\right) -\rho\left( x,x\right) \right\vert \\
		& \leq\left( \frac{\mu_{n}}{\mu_{n}^{\beta}}-\frac{\lambda_{n}^{\alpha}}{%
			\mu_{n}^{\beta}}\right) M+\frac{1}{\lambda_{n}^{\alpha}}\sum \limits_{k\in
			I_{n}}\left\vert \rho\left( x_{k},x\right) -\rho\left( x,x\right) \right\vert
	\end{align*}
	Therefore, $[V_{\rho}^{\alpha},\lambda,p]\subseteq\lbrack
	V_{\rho}^{\beta},\mu,p]$. Since (\ref{inf2}) implies (\ref{inf1}), we have
	the equality\linebreak\ $[V_{\beta}^{\beta},\mu]=[V_{\rho}^{\alpha},\lambda]$
	by $\left( i\right) $.
\end{proof}
Again, by choosing the value of $\beta$, the following two results follows
directly from the last theorem.

\begin{corollary}
Let $\lambda=\left( \lambda_{n}\right) $ and $%
\mu=\left( \mu_{n}\right) $ belong to $\Lambda$ such that $%
\lambda_{n}\leq\mu_{n}$ for all $n\in%
\mathbb{N}
_{n_{0}}$ and (\ref{inf1}) holds. Then the following statements hold:

$\left( i\right) $ $[V_{\rho}^{\alpha},\mu]\subset\lbrack V_{\rho}^{\alpha
},\lambda]$ for each $\alpha\in(0,1]$,

$\left( ii\right) $ $[V,\mu]\subset\lbrack V_{\rho}^{\alpha},\lambda]$ for
each $\alpha\in(0,1]$.
\end{corollary}

\begin{corollary}
Let $\lambda=\left( \lambda_{n}\right) $ and $%
\mu=\left( \mu_{n}\right) $ belong to $\Lambda$ such that $%
\lambda_{n}\leq\mu_{n}$ for all $n\in%
\mathbb{N}
_{n_{0}}$ and (\ref{inf2}) holds. Then the following statements hold:

$\left( i\right) $ $[V_{\rho}^{\alpha},\lambda]\subseteq\lbrack V_{\rho
}^{\alpha},\mu]$ for each $\alpha\in(0,1]$,

$\left( ii\right) $ $[V_{\rho}^{\alpha},\lambda]\subseteq\lbrack V,\mu]$ for
each $\alpha\in(0,1]$.
\end{corollary}

\begin{theorem}
Let $\alpha,\beta\in(0,1]$ be real numbers such that $%
\alpha\leq\beta$, and $\lambda=\left( \lambda_{n}\right) $, $\mu=\left(
\mu_{n}\right) \in\Lambda$ such that $\lambda_{n}\leq\mu_{n}$ for all $n\in%
\mathbb{N}
_{n_{0}}$. Then the following statements hold:

$\left( i\right) $ Let (\ref{inf1}) holds, then if a sequence is $[V_{\rho
}^{\beta},\mu]$-summable of order $\beta$, to $x\in X$, then it is $S_{\rho
}^{\alpha}(\lambda,X)$-statistically convergent of order $\alpha$, to $x$.

$\left( ii\right) $ Let (\ref{inf2}) holds, then if a sequence is $S_{\rho
}^{\alpha}(\lambda,X)$-statistically convergent of order $\alpha$, to $x\in
X $, then it is $[V_{\rho}^{\beta},\mu]$-summable of order $\beta$, to $x$.
\end{theorem}

\begin{proof}
	$\left( i\right) $ For any sequence $\left( x_{k}\right)
	\subset X$ and $\varepsilon>0$, we have%
	\begin{align*}
		& \sum_{k\in J_{n}}\left\vert \rho\left( x_{k},x\right) -\rho\left(
		x,x\right) \right\vert \\
		& =\sum_{\substack{ k\in J_{n}  \\ \left\vert \rho\left( x_{k},x\right)
				-\rho\left( x,x\right) \right\vert \geq\varepsilon}}\left\vert \rho\left(
		x_{k},x\right) -\rho\left( x,x\right) \right\vert +\sum_{\substack{ k\in
				J_{n}  \\ \left\vert \rho\left( x_{k},x\right) -\rho\left( x,x\right)
				\right\vert <\varepsilon}}\left\vert \rho\left( x_{k},x\right) -\rho\left(
		x,x\right) \right\vert \\
		& \geq\sum_{\substack{ k\in I_{n}  \\ \left\vert \rho\left( x_{k},x\right)
				-\rho\left( x,x\right) \right\vert \geq\varepsilon}}\left\vert \rho\left(
		x_{k},x\right) -\rho\left( x,x\right) \right\vert +\sum_{\substack{ k\in
				I_{n}  \\ \left\vert \rho\left( x_{k},x\right) -\rho\left( x,x\right)
				\right\vert <\varepsilon}}\left\vert \rho\left( x_{k},x\right) -\rho\left(
		x,x\right) \right\vert \\
		& \geq\sum_{\substack{ k\in I_{n}  \\ \left\vert \rho\left( x_{k},x\right)
				-\rho\left( x,x\right) \right\vert \geq\varepsilon}}\left\vert \rho\left(
		x_{k},x\right) -\rho\left( x,x\right) \right\vert \\
		& \geq\left\vert \left\{ k\in I_{n}:\left\vert \rho\left( x_{k},x\right)
		-\rho\left( x,x\right) \right\vert \geq\varepsilon\right\} \right\vert
		.\varepsilon\geq0.
	\end{align*}
	
	and so that 
	\begin{align*}
		\frac{1}{\mu_{n}^{\beta}}\sum_{k\in J_{n}}\left\vert \rho\left(
		x_{k},x\right) -\rho\left( x,x\right) \right\vert & \geq\frac{1}{\mu
			_{n}^{\beta}}\left\vert \left\{ k\in I_{n}:\left\vert \rho\left(
		x_{k},x\right) -\rho\left( x,x\right) \right\vert \geq\varepsilon\right\}
		\right\vert .\varepsilon \\
		& \geq\frac{\lambda_{n}^{\alpha}}{\mu_{n}^{\beta}}\frac{1}{\lambda
			_{n}^{\alpha}}\left\vert \left\{ k\in I_{n}:\left\vert \rho\left(
		x_{k},x\right) -\rho\left( x,x\right) \right\vert \geq\varepsilon\right\}
		\right\vert .\varepsilon\geq0.
	\end{align*}
	Since (\ref{inf1}) holds, it means that if $\left( x_{k}\right) $ is $%
	[V_{\rho}^{\beta},\mu]$-summable of order $\beta$, to $x$, then it is $%
	S_{\rho}^{\alpha}(\lambda,X)$-statistically convergent of order $\alpha$, to 
	$x$.
	
	$\left( ii\right) $ Suppose that $S_{\rho}^{\beta}(\lambda)-\lim x_{k}=x$.
	Consequently, since $(x_{k})$ is bounded there exists some $M>0$ such that $%
	\left\vert \rho\left( x_{k},x\right) -\rho\left( x,x\right) \right\vert \leq
	M$ for all $k\in%
	\mathbb{N}
	$. Then, for every $\varepsilon>0$, we have
	
	\begin{align*}
		0 & \leq\frac{1}{\mu_{n}^{\beta}}\sum \limits_{k\in J_{n}}\left\vert
		\rho\left( x_{k},x\right) -\rho\left( x,x\right) \right\vert =\frac{1}{%
			\mu_{n}^{\beta}}\sum \limits_{k\in J_{n}-I_{n}}\left\vert \rho\left(
		x_{k},x\right) -\rho\left( x,x\right) \right\vert +\frac{1}{\mu_{n}^{\beta}}%
		\sum \limits_{k\in I_{n}}\left\vert \rho\left( x_{k},x\right) -\rho\left(
		x,x\right) \right\vert \\
		& \leq\left( \frac{\mu_{n}-\lambda_{n}}{\mu_{n}^{\beta}}\right) M+\frac {1}{%
			\mu_{n}^{\beta}}\sum \limits_{k\in I_{n}}\left\vert \rho\left(
		x_{k},x\right) -\rho\left( x,x\right) \right\vert \\
		& \leq\left( \frac{\mu_{n}-\lambda_{n}^{\alpha}}{\mu_{n}^{\beta}}\right) M+%
		\frac{1}{\mu_{n}^{\beta}}\sum \limits_{k\in I_{n}}\left\vert \rho\left(
		x_{k},x\right) -\rho\left( x,x\right) \right\vert +\frac{1}{\mu_{n}^{\beta}}%
		\sum \limits_{\substack{ k\in I_{n}  \\ \left\vert \varphi\left(
				x_{k}-l\right) \right\vert <\varepsilon}}\left\vert \rho\left(
		x_{k},x\right) -\rho\left( x,x\right) \right\vert \\
		& \leq\left( \frac{\mu_{n}}{\mu_{n}^{\beta}}-\frac{\lambda_{n}^{\alpha}}{%
			\mu_{n}^{\beta}}\right) M+\frac{M}{\lambda_{n}^{\alpha}}\left\vert \left\{
		k\in I_{n}:\left\vert \rho\left( x_{k},x\right) -\rho\left( x,x\right)
		\right\vert \geq\varepsilon\right\} \right\vert +\frac{\lambda_{n}}{\mu
			_{n}^{\beta}}\varepsilon
	\end{align*}
	for all $n\in%
	\mathbb{N}
	_{n_{0}}$. Using (\ref{inf2}) we obtain that $[V^{\beta},\lambda]-\lim
	x_{k}=x$, whenever $S_{\rho}^{\alpha}(\lambda)-\lim x_{k}=x$.
\end{proof}

Similarly, by choosing the value of $\beta$, we obtain next results.

\begin{corollary}
Let $\lambda=\left( \lambda_{n}\right) $ and $%
\mu=\left( \mu_{n}\right) $ belong to $\Lambda$ such that $%
\lambda_{n}\leq\mu_{n}$ for all $n\in%
\mathbb{N}
_{n_{0}}$ and (\ref{inf1}) holds. Then the following statements hold:

$\left( i\right) $ $[V_{\rho}^{\alpha},\mu]\subset S_{\rho}^{\beta}(\lambda)$
for each $\alpha\in(0,1]$,

$\left( ii\right) $ $[V,\mu]\subset S_{\rho}^{\alpha}(\lambda)$ for each $%
\alpha\in(0,1]$.
\end{corollary}

\begin{corollary}
Let $\lambda=\left( \lambda_{n}\right) $ and $%
\mu=\left( \mu_{n}\right) $ belong to $\Lambda$ such that $%
\lambda_{n}\leq\mu_{n}$ for all $n\in%
\mathbb{N}
_{n_{0}}$ and (\ref{inf2}) holds. Then the following statements hold:

$\left( i\right) $ $S_{\rho}^{\alpha}(\lambda)\subset\lbrack V_{\rho
}^{\alpha},\mu]$ for each $\alpha\in(0,1]$,

$\left( ii\right) $ $S_{\rho}^{\alpha}(\lambda)\subset\lbrack V,\mu]$ for
each $\alpha\in(0,1]$.
\end{corollary}



\end{document}